\documentclass[leqno,11pt,oneside]{amsart}

 \usepackage[english]{babel}

\usepackage{enumitem}
\usepackage{amsmath,amsfonts,amssymb,amsthm}
\usepackage[english]{babel}

\newtheorem*{theorem*}{Theorem}

\newtheorem{teo}{Theorem}[section]
\newtheorem{prop}[teo]{Proposition}
\newtheorem{ddef}[teo]{Definition}
\newtheorem{example}[teo]{Example}
\newtheorem{cor}[teo]{Corollary}

\newtheorem{lem}[teo]{Lemma}
\newtheorem*{cor*}{Corollary}

\newtheorem*{remark}{Remark}

\newtheorem*{lem*}{Lemma}

\newtheorem*{teorA}{Theorem A}
\newtheorem*{teorB}{Theorem B}

\newtheorem*{teorA'}{Theorem A'}

\theoremstyle{definition}

\newcommand{\val}{{\rm Val}}

\newcommand{\X}{\mathfrak{X}}
\newcommand{\tg}{{\rm T}}
\newcommand{\ord}{{\rm ord}}

\newcommand{\dr}{\partial}

\newcommand{\C}{\mathbb{C}}
\newcommand{\CC}{\mathbb{C}}
\newcommand{\R}{\mathbb{R}}
\newcommand{\Pe}{\mathbb{P}}

\newcommand{\Z}{\mathbb{Z}}
\newcommand{\Q}{\mathbb{Q}}
\newcommand{\K}{\mathbb{K}}
\newcommand{\Po}{\mathcal{P}}

\newcommand{\B}{{\mathcal{B}}}

\newcommand{\D}{\mathcal{D}}
\newcommand{\F}{\mathcal{F}}

\newcommand{\cl}[1]{\mbox{$\mathcal{#1}$}}

\newcommand{\iso}{{\rm Iso}}
\newcommand{\dic}{{\rm Dic}}

\newcommand{\sep}{{\rm Sep}}

\newcommand*\xbar[1]{ %
   \hbox{ %
     \vbox{%
       \hrule height 0.3pt 
       \kern0.35ex
       \hbox{%
         \kern-0.1em
         \ensuremath{#1}%
         \kern-0.1em
       }%
     }%
   }%
}

\newcommand*\xxbar[1]{%
   \hbox{%
     \vbox{%
       \hrule height 0.3pt 
       \kern0.4ex
       \hbox{%
         \kern-0.1em
         \ensuremath{#1}%
         \kern-0.1em
       }%
     }%
   }%
}

\begin{document}

\setcounter{section}{0}
\setcounter{teo}{0}
\setcounter{exe}{0}

\author{ Eduardo Cabrera \& Rog\'erio   Mol}
\title{Separatrices for real analytic vector fields in the plane}
 \date{}
 \address{Departamento de Matem\'atica - ICEX, Universidade Federal de Minas Gerais, UFMG}
\curraddr{Av. Ant\^onio Carlos 6627, 31270-901, Belo Horizonte-MG, Brasil.}
\email{educz100@gmail.com   $\&$ rmol@ufmg.br}


\thanks{First author   financed by a CNPq Ph.D. fellowship. Second   author  partially financed by Pronex-Faperj}
\maketitle

\begin{abstract} Let $X$ be a germ of real analytic vector field at $(\R^{2},0)$ with an algebracally isolated
singularity. We say that
$X$ is a topological generalized curve if there are no topological
saddle-nodes in its reduction of singularities. In this case, we prove that  if   either the order
$\nu_{0}(X)$  or the Milnor number
$\mu_{0}(X)$ is even, then $X$  has a formal separatrix, that is, a formal invariant curve at $0 \in \R^{2}$. This result is optimal, in the sense that these hypotheses do not assure the existence of
a convergent separatrix.
\end{abstract}


 \medskip \medskip

\section{Introduction}
A \emph{separatrix} for a germ of   real or complex   analytic vector field,  at $(\R^{2},0)$
or at $(\C^{2},0)$, is an invariant irreducible formal curve passing through the origin.
The Separatrix theorem,   by C. Camacho and P. Sad, asserts that, in the complex case,
  a convergent separatrix always exists
(\cite{camacho1982}; see also \cite{jcano1997,sebastiani1997}
for alternative proofs).

A comprehensive result of this sort is  not true in the universe  
 of real analytic vector fields.  For instance,    vector fields  of center-focus type at $(\R^{2},0)$, such as
$X = y \dr/\dr x - x \dr/\dr y$,
 do not admit  invariant  curves through $0 \in \R^{2}$, neither analytic nor  formal. In this case, the complex separatrices of their complexifications
to $(\C^{2},0)$, which
 exist   by the Separatrix theorem, have   trivial real traces.
However, the search for real separatrices
may have a
 successful outcome
 within
 specific families of vector fields, delimited, for instance, by conditions
of algebraic nature imposed on the singularity.
This is the approach of  J-.J. Risler in  \cite{Risler2001}, when
proving that, in
the family of vector fields at $(\R^{2},0)$  of
real generalized curve type, a vector field
 $X=X_{\R}$    with an
algebraically isolated singularity
 has a convergent
separatrix  if either the algebraic multiplicity $\nu_{0}(X_{\R})$ or
the Milnor number $\mu_{0}(X_{\R})$ is even. More generally, omitting the real generalized curve hypothesis,
it is also proved in \cite{Risler2001} that the evenness of either $\nu_{0}(X_{\R})$ or
  $\mu_{0}(X_{\R})$ is sufficient to assure that $X_{\R}$ has a characteristic orbit, that is, a trajectory
accumulating to $0 \in \R^{2}$ having a well defined tangent at this point. In other words, a center-focus vector field must have both $\nu_{0}(X_{\R})$ and $\mu_{0}(X_{\R})$ odd.

Let us explain some terms. We say that a germ of real analytic vector field $X_{\R}$ has an algebraically isolated singularity
if   its  coefficients   vanish at $0 \in \R^{2}$ and are  relatively prime in the ring $\R\{x,y\}$ of germs
or real analytic functions at $(\R^{2},0)$.
We say that $X_{\R}$ --- or
 the singular one-dimensional real foliation  defined by it ---
is of \emph{real generalized curve} type if there are no saddle-nodes
 in its  reduction of singularities. In this process,
the final models    are  \emph{simple} singularities --- locally defined by vector fields having non-nilpotent linear part
with real eigenvalues with ratio not in $\Q^{+}$.
Saddle-nodes are simple singularities having one zero eigenvalue (see Section \ref{section-definitions} for definitions).

A real saddle-node singularity, from the topological point of view, can be a saddle, a node or a saddle-node.
We say that a germ of real analytic vector field $X_{\R}$ at $(\R^{2},0)$ is a \emph{topological real generalized curve} if there are no topological
saddle-nodes in its reduction of singularities (Definition \ref{def-topological-saddle-node}). For   vector fields in this family
 we prove the following:

\begin{theorem*}
Let $X_{\R}$ be a germ of real analytic vector field at $(\R^{2},0)$,    with an
algebraically isolated singularity at $0 \in \R^{2}$, of
topological real  generalized curve type. If
  either its order $\nu_{0}(X_{\R})$ or
its Milnor number $\mu_{0}(X_{\R})$  is even, then $X_{\R}$ has a (possibly formal)
separatrix
\end{theorem*}

We remark that, with the hypotheses stated, we cannot ask for a convergent
separatrix.
Below, we reproduce an  example from \cite{molsanz2019}, a vector field   which  is a topological generalized curve having both $\nu_{0}$ and $\mu_{0}$ even, but admitting only one separatrix
which is  purely formal.
Note that  vector fields  of generalized curve type or of center-focus type are, in particular,
of  topological generalized curve type. Hence,
Risler's results on separatrices and characteristic orbits mentioned above
can be obtained as immediate corollaries of our theorem.

For proof purposes, our main theorem will be split in two, Theorems A and B, concerning,
respectively, the situation of  even  $\nu_{0}$ and even  $\mu_{0}$.
Their proofs rely on  a study of indices and of  the reduction of singularities
of the one-dimensional foliation induced by $X_{\R}$. At no moment the Separatrix theorem is invoked.
As in \cite{Risler2001}, our strategy is to take the complexification $X_{\C}$ of $X_{\R}$ to $(\C^{2},0)$ and use the fact that it induces a complex one-dimensional foliation that is invariant
by the canonical anti-holomorphic involution $J:(x,y)\to (\bar{x},\bar{y})$. Thus, $J$ will preserve the separatrices of $X_{\C}$ and the fixed ones    are precisely those whose real traces
define separatrices for $X_{\R}$.

Two main ingredients are used in our proof.
The first one is the \emph{tangency excess index} (Definition \ref{def-tang-exc}),   an invariant
that computes the contributions of the orders of tangencies of saddle-node singularities along
the divisor of the reduction of singularities. For  a germ of complex analytic vector field,
this is \ $C^{\infty}$-invariant \cite{molrosas1919}. In the real case, the fact that a  vector field is
a topological real generalized curve implies that its   tangency excess index is even.
The second   ingredient  is the notion of \emph{balanced divisor of separatrices}. In the complex case,
this     is a tool that computes $\nu_{0}$
(see \cite{genzmer2007} and Proposition \ref{prop-mult-balanced} below).
 For the complexification $X_{\C}$ of $X_{\R}$,
the balanced divisor of separatrices can be taken symmetric with respect to the involution $J$.
Hence,  to   prove our results, it is enough  is to show  that such a balanced divisor is supported in an odd number of separatrices.
Theorem A is then obtained by a simple consideration of Proposition \ref{prop-mult-balanced}
(see Section \ref{section-theoremA}).

In order to prove Theorem B, we   use some tools from the polar theory of complex foliations,
which  are developed in Section  \ref{section-polar}.
More specifically, we
consider the \emph{polar intersection number} --- an invariant
widely studied in \cite{corral2003,cano2019} --- and  prove
Proposition \ref{prop-polar-mu-nu}, which shows a formula
relating the    order  and the Milnor number of a complex foliation with
polar intersection numbers and tangency excess indices.  This   result,  interesting in itself,
generalizes an inequality given in Proposition  2 of  \cite{cano2019}.
Finally, in Section \ref{section-theoremB}, the application of this formula gives a proof to   Theorem B.

\section{Basic definitions and notation}
\label{section-definitions}

We consider the field $\K = \R$ or $\C$, of real or complex numbers.
Along the text, we  use the symbol $\K$  as a prefix, replacing
the adjectives ``real'' or ``complex'', and as a subscript,
indicating the field we are working with.

In $\K^{2}$ we take variables $(x,y)$. Whether they are real or complex will be made clear by the context.  Let $\K [[x,y]]$ denote the ring of formal  power series in the variables $(x,y)$,
and $\K\{x,y\} \subset \K [[x,y]]$   denote  the ring  of convergent power series, the latter identified  with the ring  of germs of $\K$-analytic 
functions   at $(\K^{2},0)$.
Denote by $\X_{\K}$ and $\Omega_{\K}$  the
space of germs of $\K$-analytic  vector fields and $\K$-analytic $1-$forms at $(\K^{2},0)$.

The canonical  \emph{complexification}  
of $\R$ is given by
\[  F_{\R}(x,y)= \sum_{i,j \geq 0} a_{i,j}x^{i}y^{j} \in \R[[x,y]] \ \ \mapsto \ \ F_{\CC}(z,w) = \sum_{i,j \geq 0} a_{i,j}x^{i}y^{j} \in \C[[x,y]].\]
It defines   inclusions $\R[[x,y]] \hookrightarrow \C[[x,y]]$ and  $\R\{x,y\} \hookrightarrow \C\{x,y\}$.
This process is extended to  $\X_{\R}$ or to $\Omega_{\R}$, by simple
substitution of   coefficients in
$\R\{x,y\}$ by their complexifications in $\C\{x,y\}$, so
 we also have inclusions $\X_{\R} \hookrightarrow \X_{\C}$
and $\Omega_{\R} \hookrightarrow \Omega_{\C}$.

Let $J:(x,y) \in \C^{2} \mapsto (\bar{x},\bar{y}) \in \C^{2}$ be the
canonical anti-holomorphic involution. We can view $\R^{2}$ as the subset of $\C^{2}$
of fixed points of $J$ and call it \emph{real trace}. We
establish, in $\C[[x,y]]$ or in $\C\{x,y\}$, the   operator
\[J^\vee: f_{\C} = \sum_{i,j\geq 0}a_{ij}x^{i}y^{j} \mapsto f_{\C}^{\vee} =
 J^\vee(f_{\C}) = \sum_{i,j\geq 0} \bar{a}_{ij}x^{i}y^{j},\]
  defined in such a way that  $f_{\C}^\vee(J(x,y)) =  \xbar{f_{\C}(x,y)}$.
Note   that
 $f_{\C} \in \C[[x,y]]$ is the complexification of
$f_{\R} \in \R[[x,y]]$ if and only if $f_{\C}^{\vee} = f_{\C}$.
By acting on the coefficients of a vector field or of a $1-$form, $J^\vee$ can be extended to $\X_{\C}$ or to $\Omega_{\C}$.
Following our notation, for $X_{\C} \in \X_{\C}$ and  $\omega_{\C} \in \Omega_{\C}$, we set
$J^\vee(X_{\C})= X_{\C}^{\vee}$ and  $J^\vee(\omega_{\C}) = \omega_{\C}^{\vee}$.
Elements of $\X_{\C}$ or of $\omega_{\C}$ in the image of the complexification map are precisely those
that are fixed by  $J^\vee$.

Let $X_{\K} \in \X_{\K}$ and write
$X_{\K}= P_{\K} \partial/\partial x+Q_{\K}\partial/\partial y$, where $P_{\K}, Q_{\K} \in \K\{x, y\}$ are assumed to be relatively prime. In a small neighborhood $U_{\K}$ of   $0\in \K^2$,
the vector field $X_{\K}$ induces
a $\K$-analytic one-dimensional foliation,
which is singular, with isolated singularity at the origin,  if and only if  $P_{\K}(0,0) = Q_{\K}(0,0)=0$.
We denote this foliation by  $\F_{\K}$, which is also defined
 by the equation $ \omega_{\K} = 0$, where $\omega_{\K} \in \Omega_{\K}$ is the dual $1-$form
$\omega_{\K}  =  P_{\K} dy - Q_{\K} dx $. We will assume
   $\F_{\C}$ to be the complexification of
the  real foliation $\F_{\R}$, meaning that $X_{\C} \in \X_{\C}$ is a vector field, invariant by $J^\vee$, that is the complexification of  $X_{\R} \in \X_{\R}$.
Thus,  $\F_{\C}$ will be $J$-invariant, meaning $J$ sends leaves of $\F_{\C}$ into leaves of $\F_{\C}$.
Besides, $\F_{\R} = \F_{\C}|_{\R^{2}}$, that is, leaves of $\F_{\R}$ are obtained by the restriction
of those of $\F_{\C}$ to the real trace $\R^{2}$.
This convention will be abandoned only in Section \ref{section-polar}, where the results presented
apply to complex analytic vector fields in general.

A \emph{separatrix}  of a $\K$-analytic foliation
$\F_{\K}$ is an invariant irreducible $\K$-formal curve. That is, it is the object $B_{\K}$ given
by a reduced parametrization in one variable
$\gamma_{\K}(t) = (x_{\K}(t),y_{\K}(t))$, where  $x_{\K},y_{\K} \in t\K[[t]] $,
defined up to right composition by formal diffeomorphisms  in one variable, satisfying the
formal relation
  \[ ( P_{\K} \circ \gamma_{\K}(t)) y_{\K}'(t)   - (Q_{\K} \circ \gamma_{\K}(t))x_{\K}'(t) = 0 .\]
When $\gamma_{\K}(t)$ is $\K$-analytic, we have a \emph{$\K$-analytic separatrix}. In this case,
 the parametrization  
 defines a germ of  geometric curve in a neighborhood of $0 \in \K^{2}$, still denoted by $B_{\K}$,  such that $B_{\K}\setminus \{0\}$ is a leaf (when $\K = \C$) or two leaves (when $\K = \R$) of $\F_{\K}$.
  We denote the
 family of all separatrices of $\F_{\K}$
  by $\sep(\F_{\K})$.

  The $J$-symmetry of $\F_{\C}$ implies that the involution $J$ takes separatrices into separatrices
by the following correspondence:
\[  B_{\C}:\ \gamma_{\C}(t) = (x_{\C}(t),y_{\C}(t)) \ \ \mapsto \ \  B_{\C}^{\vee}:\ \gamma_{\C}^{\vee}(t) = J(x_{\C}(\bar{t}),y_{\C}(\bar{t})).\]
Separatrices such that $B_{\C} = B_{\C}^{\vee}$
are called \emph{real} and we denote their family by   $\sep^{r}(\F_{\C})$.
They are in bijection with $\sep(\F_{\R})$.
The other separatrices of $\F_{\C}$, those such that
$B_{\C} \neq B_{\C}^{\vee}$,
are called   \emph{purely complex} and we denote their set   by   $\sep^{c}(\F_{\C})$.
We then have a disjoint union
$\sep(\F_{\C}) = \sep^{r}(\F_{\C}) \cup \sep^{c}(\F_{\C})$.
Since $(B_{\C}^{\vee})^{\vee} = B_{\C}$, we have
 that separatrices in $\sep^{c}(\F_{\C})$ appear in pairs.
This remark will be crucial in  the development of our results.

For a germ of $\C$-analytic foliation $\F_{\C}$, induced by  vector field $X_{\C}$, we consider the following invariants:
\begin{itemize}
\item  $\nu_{0}(\F_{\C})$ is the \emph{order} or the \emph{algebraic multiplicity} of   $\F_{\C}$ at $0 \in \C^{2}$, which is the minimum of the orders of the coefficients of $X_{\C}$.
\item $\mu_{0}(\F_{\C})$, the \emph{Milnor number} of   $\F_{\C}$, defined as  the dimension of the local
algebra
\begin{equation*}
\mu_{0}(\F_{\C})= \dim_{\mathbb{C}} \frac{\mathbb{C}\{x, y\}}{(P_{\CC}, Q_{\CC})}.
\end{equation*}
\end{itemize}
For a germ of real foliation $\F_{\R}$ with complexification $\F_{\C}$, we set
$\nu_{0}(\F_{\R}) = \nu_{0}(\F_{\C})$ and  $\mu_{0}(\F_{\R}) = \mu_{0}(\F_{\C})$.

We denote by $\pi_{\K}:  (\tilde{\K^{2}},D_{\K}) \to (\K^{2},0)$ the  $\K$-\emph{blow-up} at  $0 \in \K^{2}$.
The space $\tilde{\K^{2}}$
  is  a   $\K$-analytic surface,
 obtained by replacing the origin by its set of tangent
directions,  parametrized  by the projective line  $D_{\K} \simeq \Pe^{1}_{\K}$.
It is defined by coordinates $(x,u), (v,y) \in \K^{2}$
identified by the relations $y = ux$ and $x = vy$, in such a way that
$\pi_{\K}:(x,u) \mapsto (x,ux),\ (v,y) \mapsto (vy,y)$.
Restricted to  $\tilde{\K^{2}} \setminus D_{\K}$, the blow-up is a $\K$-analytic diffeomorphisms   onto its image.

When $\K = \C$, the involution $J:(x,y) \in \C^{2} \mapsto (\bar{x},\bar{y}) \in \C^{2}$   lifts to
$\tilde{\C^{2}}$,  defining a unique continuous involution  $J_{1}:  \tilde{\C^{2}} \to \tilde{\C^{2}}$
such that   $\pi_{\C} \circ J_{1} = J  \circ \pi_{\C}$. We can canonically identify
$\tilde{\R^{2}}$ with the fixed set of
  this lifting, that is
$\tilde{\R^{2}}  = \tilde{\C^{2}} \cap  J_{1}(\tilde{\C^{2}})$, so
that $D_{\R} = D_{\C} \cap \tilde{\R^{2}}$ and
$\pi_{\R} = \pi_{\C}|_{\tilde{\R^{2}}}$. We will say that
$\tilde{\R^{2}}$ is the \emph{real trace} of  $\tilde{\C^{2}}$.

Given     a germ
of $\K$-analytic singular foliation $\F_{\K}$ at $(\K^{2},0)$,     there is a unique $\K$-analytic singular  foliation
$\pi_{\K}^{*} \F_{\K}$, the \emph{strict transform} of $\F_{\K}$,  defined in $(\tilde{\K^{2}},D_{\K})$ and
 having isolated singularities over $D_{\K}$, that corresponds diffeomorphically to $\F_{\K}$ over the points where $\pi_{\K}$ is a local diffeomorphism. If the   line $D_{\K}$ is  invariant by
$\pi_{\K}^{*} \F_{\K}$, we   say that the blow-up $\pi_{\K}$ or the component $D_{\K}$ is   \emph{non-dicritical}. It is \emph{dicritical} otherwise.

We iterate the process of blowing up and consider
a   composition $\sigma_{\K} = \pi^{1}_{\K} \circ \cdots \circ \pi^{k}_{\K}$ of blow-ups
$\pi^{j}_{\K}: (\tilde{M}^{j}_{\K},\D^{j}_{\K}) \to (\tilde{M}^{j-1}_{\K},\D^{j-1}_{\K})$, for $j=1,\ldots,k$,
  where:
  \begin{itemize}
  \item $(\tilde{M}^{0}_{\K},\D^{0}_{\K}) \simeq (\K^{2},0)$, $(\tilde{M}^{1}_{\K},\D^{1}_{\K}) \simeq (\tilde{\K^{2}}, D_{\K})$ and
 $\pi^{1}_{\K}: (\tilde{M}^{1}_{\K},\D^{1}_{\K}) \to (\tilde{M}^{0}_{\K},\D^{0}_{\K})$ is the
 $\K$-blow-up at $0 \in \K^{2}$;
 \item each $\pi^{j}_{\K}$ is a blow-up at a point $p_{j-1} \in \D_{\K}^{j-1}$;
 \item For $j=1,\ldots,k$, $\D^{j}_{\K} =  (\pi^{1}_{\K} \circ \cdots \circ \pi^{j}_{\K})^{-1}(0)$ is a normal crossings divisor, whose irreducible
 components are isomorphic to projective lines $\Pe^{1}_{\K}$.
 \end{itemize}
Setting $(\tilde{M}_{\K},\D_{\K}) \simeq  (\tilde{M}^{k}_{\K},\D^{k}_{\K})$,
we  denote this sequence of blow-ups by $\sigma_{\K}: (\tilde{M}_{\K},\D_{\K}) \to (\K^{2},0)$.
 The smooth points of $\D_{\K}$ are called \emph{trace points}, while its singular points are called \emph{corners}.
This iterated  construction, applied to a germ of singular foliation $\F_{\K}$,  produces a {strict transform} foliation $\sigma_{\K}^{*}  \F_{\K} $ in $(\tilde{M}_{\K},\D_{\K})$.
The irreducible components of $\D_{\K}$ are   classified
 as \emph{dicritical} --- non-invariant by $\tilde{\F}_{\K}$ ---  or \emph{non-dicritical} ---
 invariant by $\tilde{\F}_{\K}$.

Suppose that $0 \in \K^{2}$ is a singularity for $\F_{\K}$.  This  singularity is said to be \emph{simple} if
the linear part $DX_{\K}(0)$ of the vector field $X_{\K}$, that induces $\F_{\K}$,  has
eigenvalues $\lambda_ 1, \lambda_2 \in \K$ satisfying  one of the two conditions:
	\begin{enumerate}[label=\roman*)]
		\item $\lambda_1,   \lambda_2 \neq 0$ and $ \lambda_1 / \lambda_2  \notin \mathbb{Q}^+$;
		\item $\lambda_1 \neq 0$, $\lambda_2 = 0$.
	\end{enumerate}
In case (i) we have a \emph{non-degenerate} singularity and in case (ii) we have a \emph{saddle-node} singularity.
In the real case, for reasons that will become clear later, we will sometimes refer to the second case as an
\emph{algebraic saddle-node}.
A simple singularity has exactly two smooth transversal
separatrices (see \cite{ilyashenko2008}). In the non-degenerate case, both are convergent. In the
saddle-node case, the one associated to the eigenspace of the non-zero eigenvalue is convergent
and is called \emph{strong separatrix}.  The other separatrix is, in  principle,
formal, and is called \emph{weak} of \emph{central separatrix}.
A saddle-node singularity can be expressed, in formal coordinates, by a $1-$form of the type \cite{martinet1982}:
\begin{equation}
\label{eq-saddle-node-normal}
\omega_{\K} = y(1+\mu x^k) dx + x^{k+1} dy,
\end{equation}
where $\mu \in \K$  and $k \in \mathbb{Z}_{+}$ are formal invariants.
In this writing, the strong separatrix corresponds to $\{x=0\}$ and  the weak separatrix to $\{y=0\}$. The integer $\iota^{w}_0(\mathcal{F}_{\K})= k+1>1$ is called  \emph{weak index} of the saddle-node.

When $\K = \R$, an algebraic saddle-node singularity can be, from the topological point of view, a saddle, a node or a saddle-node. If $\iota^{w}_0(\mathcal{F}_{\K})$ is even, it is a topological saddle-node.
If $\iota^{w}_0(\F_{\K})$ is odd,  it can be either a saddle or a node (see \cite[Th. 9.1]{ilyashenko2008}).

The foliation $\F_{\K}$
admits   a \emph{ reduction of singularities} or \emph{desingularization}   by Seidenberg's theorem \cite{seidenberg1968}.
This means that there is a sequence of $\K$-blow-ups
\begin{equation}
\label{eq-reduction-of-sing}
\sigma_{\K} =  (\tilde{M}_{\K},\cl{D}_{\K})   \to (\K^{2},0)
\end{equation}
such that strict transform foliation
 $\tilde{\F}_\K = \sigma_{\K}^{*} \F$ has a finite number o singularities, all of them over $\cl{D}_{\K}$,
 all of them   simple.  We  can   ask further   $\tilde{\mathcal{F}}_\mathbb{C}$ to be everywhere transversal to each dicritical component of
 $\D_{\C}$   ---
 we say in this case that $\sigma_{\K}$ desingularizes the set of separatrices
$\sep(\F_{\K})$ (see \cite{camacho1984}).
Also, we require  that  the sequence of blow-ups $\sigma_{\K}$ is minimal with these properties,
being thus unique up to  isomorphism. Once and for all,
we fix a minimal reduction of singularities \eqref{eq-reduction-of-sing} for $\F_{\K}$.

We say that $B_{\K} \in \sep(\F_{\K})$ is a \emph{dicritical separatrix} if $\tilde{B}_{\K} = \sigma_{\K}^{*}B_{\K}$ touches $\cl{D}_{\K}$
in a dicritical component. Otherwise, we say that $B_{\K}$ is an \emph{isolated} separatrix.
This engenders a decomposition
$\sep(\F_{\K}) =  \iso(\F_{\K}) \cup \dic(\F_{\K})$, where notation is self-explanatory.
Isolated separatrices are in one-to-one correspondence with trace singularities of $\tilde{\F}_{\K}$.
All purely formal separatrices of $\F_{\K}$ are in $\iso(\F_{\K})$ and come from weak separatrices
of saddle-node singularities of   $\tilde{\F}_{\K}$ that are not contained in $\D_{\K}$.

When $\K = \C$, the $J$-symmetry of $\F_{\C}$
 allows us to
lift the involution
$J$ to a   continuous involution  $\tilde{J} :  \tilde{M}_{\C} \to \tilde{M}_{\C}$
such that   $\sigma_{\C} \circ \tilde{J } = J  \circ \sigma_{\C}$.
In this way $\tilde{M}_{\R}$ can be identified with the fixed set of $\tilde{J}$.
The foliation $\tilde{\F}_{\C}$ is evidently   symmetric with respect to $\tilde{J}$ and
$\tilde{\F}_{\R}$ is the restriction of $\tilde{\F}_{\C}$ to the real trace $\tilde{M}_{\R}$.
Also, $\sigma_{\R}$ is the restriction
of $\sigma_{\C}$  to $\tilde{M}_{\R}$.

We have the following concept \cite{camacho1984, Risler2001}:
a germ of $\K$-analytic singular foliation $\F_{\K}$ at $(\K^{2},0)$ is a \emph{$\K$-generalized curve}
(a $\K$GC foliation) if there are no saddle-nodes in its reduction of singularities.
When $\K = \R$, considering that  algebraic saddle-nodes may differ
in their topological picture, we delimit the following broader family of foliations:
\begin{ddef}
\label{def-topological-saddle-node}
{\rm
A germ of $\R$-analytic singular foliation $\F_{\R}$ at $(\R^{2},0)$ is a \emph{topological $\R$-generalized curve}
if there are no topological saddle-nodes in its reduction of singularities.
}\end{ddef}

The  notion of generalized curve foliation can be weakened by the admission, in the reduction of singularities, of some
saddle-node singularities with a good orientation.
A saddle-node singularity  of $\tilde{\F}_{\K}$ is said to be \emph{tangent} if
its weak separatrix is contained in $\D_{\K}$.
For instance, a saddle-node   in a corner of $\D_{\K}$ is always tangent.
We say that $\F_{\K}$ is of
\emph{$\K$-second type} (a $\K$ST foliation)  if there are no tangent  saddle-nodes in its reduction of singularities \cite{mattei2004}.
The family of $\K$ST  foliations evidently contains that of $\K$GC foliations.

To each $D_{\K} \subset \D_{\K}$ we associate a \emph{weight}
 $\rho(D_{\K}) \in \Z_{>0}$,
defined as the order at $0 \in \K^{2}$ of a germ of $\K$-analytic curve  whose strict transform by $\sigma_{\K}$ is   transversal to $\D_{\K}$. The weight also has
 a combinatorial description which can be seen in \cite{camacho1984}.
It appears in the definition of the  following invariant, which  is a  measure of  the existence of tangent saddle-nodes in the reduction of singularities of $\F_{\K}$ :
\begin{ddef}
\label{def-tang-exc}
{\rm
The	$\K$-\emph{tangency excess} of $\F_{\K}$ at $0 \in \K^{2}$ is the integer
\[  \tau_{0}(\F_{\K})
=\sum_{q\in {\rm SN}(\tilde{\F}_{\K})} \rho(D_{\K,q})(\iota_q^{w}(\tilde{\F}_{\K})-1).\]
}\end{ddef}
In the formula above, ${\rm SN}(\tilde{\F}_{\K})$ stands for the set of tangent saddle-nodes of $\tilde{\F}_{\K}$ in $\D_{\K}$, $D_{\K,q}$ is the component of $\D_{\K}$ containing the weak separatrix of $\tilde{\F}_{\K}$  at $q$   and $\iota_q^{w}( \tilde{\F}_{\K})$ is the weak index.
When  $\F_{\K}$ is has a simple singularity at $0 \in \K^{2}$, the sum is understood as empty and
$\tau_{0}(\F_{\K})=0$.
Observe that $\tau_{0}(\F_{\K}) \geq 0$ and  $\tau_{0}(\F_{\K}) = 0$ if and only if
${\rm SN}(\tilde{\F}_{\K}) = \emptyset$, i.e. if and only if $\F_{\K}$ is
a $\K$-second type foliation.

Our main theorems will be stated for the family of topological
 $\R$GC  foliations. The following remark will be useful in their proofs:
\begin{remark}
{\rm
When $\K = \R$ and  $\F_{\R}$ is a topological $\R$GC foliation, then
$\tau_{0}(\F_{\R})$ is even. Indeed,   all tangent saddle-node are topological
saddles or nodes, so their weak indices are odd and the  contribution of each one to
the tangency excess is even.
}\end{remark}


\section{The algebraic multiplicity and separatrices}
\label{section-theoremA}

For $\K=\R$ and $\C$, let    $X_{\K}$ denote  germs of $\K$-analytic vector fields
at $(\K^{2},0)$ defining a foliation $\F_{\K}$. We keep in this section the convention that   $\F_{\C}$
is the complexification of $\F_{\R}$. For them, we fix  reductions of singularities as
in \eqref{eq-reduction-of-sing}.

The \emph{valence} $\val(D_{\K})$ of    $D_{\K} \subset \D_{\K}$  is the number of irreducible components of $\D_{\K}$ that intercept  $D_{\K}$, other than $D_{\K}$ itself.
Observe that, if $D_{\C} = D_{\C}^{\vee}$ and
$D_{\R} = D_{\C} \cap \tilde{M}_{\R}$,
then $\val(D_{\R}) \leq \val(D_{\C})$, and the inequality   can be strict.
However, from the $\tilde{J}$-symmetry of $\tilde{\F}_{\C}$, they have
the same parity: $\val(D_{\R}) \equiv \val(D_{\C}) \pmod{2}$.
On the other hand, it is easy to see that
 $\rho(D_{\R}) = \rho(D_{\C})$.

A \emph{$\K$-divisor of separatrices } for   $\mathcal{F}_{\K}$  is a formal sum
\[\mathcal{B}_{\K}=\sum_{B_{\K} \in  \sep(\F_{\K})}a_{B_{\K}} \cdot B_{\K},\]
where the coefficients $a_{B_{\K}} \in \mathbb{Z}$ are zero except for a finite number of $B_{\K}\in   \sep(\F_{\K})$.
The \emph{order}  of   $\mathcal{B}_{\K}$ is calculated
additively:
$\nu_0 (\B_{\K}) = \sum_{B_{\K} \in  \sep(\F_{\K})}a_{B_{\K}} \nu_{0}(B_{\K})$.
The \emph{support} of $\B_{K}$   is the set formed by all $B_{\K} \in  \sep(\F_{\K})$
such that $a_{B_{\K}} \neq 0$.
Whenever necessary, we can separate isolated and dicritical separatrices and decompose $\B_{\K} = \B_{\K,\text{iso}} + \B_{\K,\text{dic}}$.

We say that $\B_{\K}$ is
a	\emph{$\K$-balanced divisor of separatrices}  for $\F_{\K}$  (see \cite{genzmer2007})
 if its  coefficients
satisfy the following conditions:
\begin{itemize}
\item  $a_{B_{\K}} \in \{-1, 0, 1\}$;
\item $a_{B_{\K}}=1$ for every $B_{\K} \in \iso(\F_{\K})$;
\item for a fixed dicritical component $D_{\K} \subset \D_{\K}$, the following equality holds:
$$\sum_{B_{\K} \in \sep(D_{\K})}  a_{B_{\K}} = 2 - \val(D_{\K}),$$
where $\sep(D_{\K})$ is the set of separatrices associated to $D_{\K}$.

\end{itemize}

Observe that  $\nu_{0}(\B_{\K})$ is the same for all balanced divisors associated to
 $\F_{\K}$.
Actually, in the complex case, it has the following relation with the order $\nu_{0}(\F_{\C})$:
\begin{prop} \cite{genzmer2007}
\label{prop-mult-balanced}
Let $\mathcal{F}_\mathbb{C}$ be a germ of singular foliation at $(\mathbb{C}^2,0)$ having  $\mathcal{B}_\mathbb{C}$ as a balanced divisor of separatrices.  Then
\[\nu_0(\mathcal{F}_\mathbb{C})	= \nu_0(\mathcal{B}_\mathbb{C})-1+\tau_{0} ( \mathcal{F}_\mathbb{C}).\]
\end{prop}
Since $\tau_{\C,0}(\F_{\C}) \geq 0$, with equality if and only if 	$\mathcal{F}_\mathbb{C}$ is a $\C$-second type   foliation, the proposition can be read in the following way:
$$\nu_0(\mathcal{B}_\mathbb{C}) \leq \nu_0(\mathcal{F}_\mathbb{C})	+1,$$
with equality happening
if and only if 	$\mathcal{F}_\mathbb{C}$ is  $\C$ST.

Taking into account again the $\tilde{J}$-symmetry of $\tilde{\F}_{\C}$, we can easily see that
if $D_{\C} \subset \D_{\C}$ is a dicritical component, then $D_{\C}^\vee= \tilde{J}(D_{\C})$
is also a dicritical component and $\val (D_{\C}) = \val (D_{\C}^\vee)$. It is also clear that
if $D_{\C}^{\vee} =  D_{\C}$, then $D_{\R} =  D_{\C} \cap \tilde{M}_{\R}$
is a dicritical component for the   reduction of singularities $\sigma_{\R}$.
We can therefore
 produce
a balanced divisor of separatrices $\B_{\C}$ with the following two  conditions:
\begin{enumerate}[label=\roman*)]
\item   $B_{\C}$ and $B_{\C}^{\vee} = J(B_{\C})$
have the same coefficient for every $B_{\C} \in \sep(\F_{\C})$;
\item if $D_{\C} \subset \D_{\C}$ is a dicritical component such that $D_{\C}= D_{\C}^\vee$, then every $B_{\C} \in \sep(D_{\C})$ with $a_{B_{\C}} \neq 0$ is in $\sep^{r}(\F_{\C})$.
\end{enumerate}
Keeping coherence with our terminology, we call such a balanced
divisor   \emph{$J$-symmetric}.
We have a decomposition
\begin{equation}
\label{eq-decompostion-balanced}
\B_{\C} = \B_{\C}^{r} + \B_{\C}^{c},
\end{equation}
where $\B_{\C}^{r}$ comprises all separatrices in $\sep_{\C}^{r}(\F_{\C})$ and
$\B_{\C}^{c}$ those in $\sep_{\C}^{c}(\F_{\C})$. By means of the identification
 $\sep^{r}(\F_{\C}) \simeq \sep(\F_{\R}) $,
the divisor  $\B_{\C}^{r}$ corresponds to a  divisor of separatrices for $\F_{\R}$.
It is not, in principle,
 an $\R$-balanced divisor of separatrices for $\F_{\R}$,
since the real and complex valences of corresponding dicritical  components may differ. However, given a $J$-symmetric balanced divisor $\B_{\C}$ for
$\F_{\C}$, we can produce, for $\F_{\R}$, a balanced divisor
$\B_{\R}$ such that every non-zero coefficient of   $B_{\R} \in \sep(\F_{\R})$ coincides with that
of the corresponding   $B_{\C} \in \sep(\F_{\C})$ in $\B_{\C}$.
Abusing terminology, we will say that $\B_{\R}$ is \emph{contained} in $\B_{\C}$.

\begin{prop}
\label{prop-equal-mod2}
Let $X_{\R}$ be a germ of real analytic vector field at $(\R^{2},0)$, $X_{\C}$ be its complexification, $\F_{\R}$ and $\F_{\C}$ be the foliations associated to them. If $\B_{\R}$ and $\B_{\C}$ are balanced divisors of separatrices for $\F_{\R}$ and $\F_{\C}$, then
\begin{equation*}
	\nu_0 (\mathcal{B_\mathbb{R}}) = \nu_0 (	\mathcal{B_\mathbb{C}})     
	\pmod{2}.
\end{equation*}
\end{prop}
\begin{proof}
 We can suppose   that $\B_{\C}$ is $J$-symmetric and that $\B_{\R}$ is contained in $\B_{\C}$.
Denoting $\iso^{r}(\F_{\C}) = \iso(\F_{\C}) \cap \sep^{r}(\F_{\C})$ and
$\iso^{c}(\F_{\C}) = \iso(\F_{\C}) \cap \sep^{c}(\F_{\C})$,
we have a splitting:
\[ \B_{\C,\text{iso}} =  \sum_{B_{\C} \in \rm \iso^{r}(\F_{\C})} a_{B_{\C}} \cdot  B_{\C} +  \sum_{B_{\C} \in \rm \iso^{c}(\F_{\C})} a_{B_{\C}} \cdot B_{\C} = \B_{\C, \text{iso}}^{r} + \B_{\C, \text{iso}}^{c} . \]
For symmetry reasons,    $\nu_{0}( \B_{\C , {\rm iso}}^{c}  ) $ is even.
On the other hand, separatrices    in $\iso^{r}(\F_{\C})$ are in one to one correspondence
with those in $\iso(\F_{\R})$.
Therefore
$\nu_{0}(\B_{\C , \text{iso}}^{r}) = \nu_{0}(\B_{\R , \text{iso}})$, allowing us to conclude
that
\begin{equation}
\label{eq-nu-div-iso}
 \nu_{0}(\B_{\C, \text{iso}}) = \nu_{0}(\B_{\R, \text{iso}}) \pmod{2}.
\end{equation}
Regarding $\B_{\C,\text{dic}}$,
we decompose, as before,
$ \B_{\C,\text{dic}} =
\B_{\C, \text{dic}}^{r} + \B_{\C, \text{dic}}^{c} $. We
also separate
the dicritical components of $\D_{\C}$ in $\dic(\D_{\C}) = \dic^{r}(\D_{\C}) \cup  \dic^{c}(\D_{\C})$, according to the invariance by the involution $\tilde{J}$.
Due to  symmetry,
if $D_{\C} \in \dic^{c}(\D_{\C})$, then $D_{\C}^{\vee} = \tilde{J}(D_{\C}) \in \dic^{c}(\D_{\C})$
and $\val(D_{\C}) = \val(D_{\C}^{\vee})$. This implies that $\nu_{0}(\B_{\C, \text{dic}}^{c})$ is even.
On the other hand, for   $D_{\C} \in \dic^{r}(\D_{\C})$ and $D_{\R}  = D_{\C} \cap \R^{2}$,  we have,
as noted before, $\val(D_{\C}) \equiv  \val(D_{\R}) \pmod{2}$.
As a consequence,
\begin{equation}
\label{eq-nu-div-dic}
 \nu_{0}(\B_{\C, \text{dic}}) = \nu_{0}(\B_{\R, \text{dic}}) \pmod{2}.
 \end{equation}
Equations \eqref{eq-nu-div-iso} and \eqref{eq-nu-div-dic} prove the result.
\end{proof}


We observe  that, since  $\tilde{\F}_{\C}$  is $\tilde{J}$-symmetric, tangent saddle nodes outside the real trace appear in pairs and have the
 same weak indices. Thus,
\begin{equation}
\label{eq-tau-mod2}
 \tau_{0} (\F_{\C}) \equiv \tau_{0} (\F_{\R}) \pmod{2}.
\end{equation}
This is an ingredient for the next
 proposition, which gives sufficient conditions
for the existence of separatrices for germs of real analytic vector fields:
\begin{prop}
\label{prop-parity-separtrix}
 Let $X_{\R}$ be a germ of real analytic vector field at $(\R^{2},0)$  inducing a foliation $\F_{\R}$.
Suppose that the algebraic multiplicity
$\nu_{0}(\F_{\R})$ and the tangency excess $\tau_{0}(\F_{\R})$ have the same parity.
Then $\F_{\R}$ admits a
formal separatrix.
\end{prop}
\begin{proof} We have, by \eqref{eq-tau-mod2}, that $\tau_{0} (\F_{\C})$ and  $\tau_{0} (\F_{\R})$
have the same parity.
Since $\nu_{0}(\F_{\C})$ and $ \nu_{0}(\F_{\R})$ are equal by definition, the proposition's hypothesis gives that
$\nu_{0}(\F_{\C})$ and  $\tau_{0}(\F_{\C})$ have the same parity.
By Proposition \ref{prop-mult-balanced}, we then have that
$\nu_{0}(\B_{\C}) = 1 \pmod 2$. Finally, by Proposition \ref{prop-equal-mod2},
$\nu_{0}(\B_{\R}) = 1 \pmod 2$. This gives that $\B_{\R}$ has non-empty support, assuring   the existence of a separatrix for $\F_{\R}$.
\end{proof}

The main result of this section comes as consequence of this proposition:
\begin{teorA} Let $X_{\R}$ be a germ of real analytic vector field at $(\R^{2},0)$ with associated foliation
$\F_{\R}$. Suppose that $\F_{\R}$ is a real topological generalized curve foliation.
If $\nu_{0}(\F_{\R})$ is even,
then $\F_{\R}$ has a formal separatrix.
\end{teorA}
\begin{proof} This is a consequence of the fact that $\tau_{0}(\F_{\R})$ is even for a real
topological generalized curve foliation   $\F_{\R}$.
\end{proof}

Observing that an  $\R$GC  foliation is also a topological $\R$GC foliation,   and
that all its separatrices are analytic,
we can recover the following result from   \cite{Risler2001}:
\begin{cor}
\label{cor-nu-RGC}
If  $\F_{\R}$ is an $\R$-generalized curve foliation with $\nu_{0}(\F_{\R})$ is even,
then  it admits an
analytic separatrix.
\end{cor}

A germ of non-dicritical foliation $\F_{\R}$ at $(\R^{2},0)$ is of \emph{center-focus} or
of \emph{monodromic} type if
the following equivalent conditions happen (see, for instance, \cite[Th. 9.13]{ilyashenko2008}):
\begin{itemize}
\item  for every  smooth $\R$-analytic semi-branch  $\Gamma^{+}$ at $(\R^{2},0)$ there is
a first return map $\rho: \Gamma^{+} \to \Gamma^{+}$ (the \emph{monodromy map});
\item  there are no \emph{characteristic orbits}, i.e., leaves that accumulate to $0 \in \R^{2}$
with   well defined tangent at the origin;
\item the reduced model  $\tilde{\F}_{\R}$ has no trace singularities --- and thus $\sep(\F_{\R})$ is empty --- and all corner singularities
are topological saddles.
\end{itemize}
The third of the above characterizations means, in particular,  that a center-focus foliation
is a topological $\R$GC foliation.
The following fact is then an easy consequence of Theorem A
(see \cite[Cor. 3.5]{Risler2001}):
\begin{cor}
\label{cor-nu-centerfocus}
If  $\F_{\R}$ is a center-focus foliation, then $\nu_{0}(\F_{\R})$ is odd.
\end{cor}
\begin{proof}
If $\nu_{0}(\F_{\R})$ were even, there would be a formal separatrix, giving a contradiction.
\end{proof}

\section{Local polar invariants of complex analytic foliations}
 \label{section-polar}

The main goal of this section is to prove Proposition \ref{prop-polar-mu-nu}.
It permits the calculation of the sum of the algebraic multiplicity and of the Milnor number
of a complex foliation in terms of an invariant of polar nature.
Our result   generalizes, for dicritical foliations, an inequality shown in
  \cite[Prop. 2]{cano2019}.
The results in this section apply to germs of $\C$-analytic foliations in general.
All objects here are complex and we permit ourselves to  lighten the notation
by omitting  the subscript ``$\C$''.

Let $X = P(x,y) \partial / \partial x + Q(x,y) \partial / \partial y$ be a germ of
complex analytic vector field at $(\C^{2},0)$
 and $\omega = P(x,y) dy- Q(x,y)dx$ be its dual $1-$form. Denote by  $\mathcal{F}$ be  the germ of foliation induced by them.
The \textit{polar curve}  of $\mathcal{F}$ with respect to $(a:b) \in \mathbb{P}^1$, where
$\mathbb{P}^1$ is the complex  projective line, is
the analytic curve $\Po_{(a:b)}  $ defined by the equation $aP-bQ=0$.
It consists of the sets of points where $X$ has inclination $a/b$ (when $b \neq 0$).
For generic $(a:b) \in \mathbb{P}^1$, the curves $\Po_{(a:b)}$ are reduced and
equisingular. These curves are called \emph{generic polar curves}.
For instance, it is evident that an $\F$-invariant component of $\Po_{(a:b)}$ must be a line through the origin
of inclination $a/b$. Thus, except for the trivial radial foliation, the generic polar curve $\Po_{(a:b)}$ is free from  $\F$-invariant    components. It is also clear that the equality
  $\nu_{0}(\Po_{(a:b)}) =  \nu_{0}(\F)$ is   true for a generic polar curve.
We refer the reader to \cite{corral2003,cano2019} for a  more extensive discussion on polar curves and polar invariants of foliations.

%

Let $B$ be a germ of separatrix of $\F$ at  $(\C^{2},0)$,   with  parametrization
$\gamma_{B}(t)$. The \emph{polar intersection} of $\F $ with respect to $B$ (see  \cite{cano2019,genzmer2018})  is defined as
$$(\Po_{(a:b)} , B)_{0}= \ord_{t=0}\left((aP-bQ) \circ  \gamma_{B}(t)\right),$$
 where
$(\, \cdot \, , \, \cdot \, )_{0}$ stands for the intersection number of germs of complex formal curves
at $(\C,0)$
and $\Po_{(a:b)}$ is a generic polar curve.
We denote this number by $p_{0}(\F,B)$.
More generally, if $\B=\sum_{B \in  \sep(\F)}a_{B} \cdot B$ is a divisor of
separatrices of $\F$, we can define \emph{polar intersection} of $\F$ with respect to $\B$
in an additive way:
\[   p_{0}(\F,\B)   = \sum_{B \in  \sep(\F)} a_{B} \, p_{0}(\F,B) .\]
This definition will be used specifically in the case where $\B$ is a balanced divisor of separatrices
for  $\F$ at $0 \in \C^{2}$.

Let $\Gamma$ be a  formal branch at $(\C^{2},0)$,  non-invariant by $\F$, with  parametrization $\gamma(t)$.
The \emph{tangency order} of  $\F$ with respect to $\Gamma$ at $0 \in \C^{2}$ is the integer
\[ \tg_{0}(\F,\Gamma) = {\rm ord}_{t=0} \gamma^{*} \omega.\]
To a complex blow-up  $\pi :(\tilde{\C^{2}}, D  ) \to (\C^{2}, 0)$, we
 associate the following integer:
\begin{equation}
\label{eq-def-m}
 m=
\begin{cases}
\nu_{0}(\F) \ \ \text{if $\pi$ is non-dicritical} \medskip \\
\nu_{0}(\F)+1 \ \ \text{if $\pi$ is  dicritical}.
\end{cases}
\end{equation}
 Denote by $\tilde{\Gamma}  = \pi^{*} \Gamma$ and by  $\tilde{\F}  = \pi ^{*} \F$
  the strict transforms of   $\Gamma$ and
  $\F$, and set $q =  \tilde{\Gamma} \cap D$.
  The tangency orders $\tg_{0}(\F,\Gamma)$  and $\tg_{q}(\tilde{\F} ,\tilde{\Gamma})$
  are related in the following way:

\begin{prop}
\label{prop-blowup-tau}
Let  $\Gamma$ be a formal branch at $(\C^{2}, 0)$, non-invariant by the foliation $\F$.
Then
\[\tg_{0}(\F,\Gamma) = m \nu_{0}(\Gamma)  +  \tg_{q}(\tilde{\F} ,\tilde{\Gamma}),\]
where $m$, $\tilde{\F}$ , $\tilde{\Gamma}$ and $q$   are defined as above.
\end{prop}
\begin{proof}
 Let us write $\gamma(t)=(x(t), y(t))$. By a linear change of coordinates, we can suppose that the $x$-axis is   the tangent cone of $\Gamma$, so that $\nu_{0}(\Gamma) = \text{ord}_{t=0} x(t)$.
 Write, in coordinates, the blow-up map as $\pi(x,u) = (x,ux)$, so that $q = (0,0)$.
 Thus, we can obtain a  parametrization  $\tilde{\gamma}(t) = (x(t),u(t))$ of $\tilde{\Gamma}$ such that
 $\gamma (t)=\pi \circ \tilde{\gamma}  (t)$, which is equivalent to
 $y(t) = u(t) x(t)$.
If $\omega=P(x,y)dy-Q(x,y)dx$, then
$$\pi^*\omega =  (-Q(x,ux)+uP(x,ux))dx+xP(x,ux)du .$$
The foliation $\tilde{\F}$ is then defined by $\tilde{\omega}  = (1/x^{m}) \pi^*\omega$,
where $m$ is given by \eqref{eq-def-m}.
We have
$ \gamma^{*} \omega = (\pi\circ \tilde{\gamma})^{*} \omega =  \tilde{\gamma}^* \pi^*\omega$.
Thus
\begin{eqnarray*}
\tg_{0}(\F,\Gamma) =  \text{ord}_{t=0} \{\gamma^{*} \omega\} =
\text{ord}_{t=0}\{\tilde{\gamma}^* \pi^*\omega\}& = &
\text{ord}_{t=0}\{(x(t))^{m} \tilde{\gamma}^*\tilde{\omega}\} \\
&= & m \, \text{ord}_{t=0}x(t) +  \text{ord}_{t=0}\{\tilde{\gamma}^* \tilde{\omega} \} \\
&=  & m \nu_{0}(\Gamma) +  \tg_{q}(\tilde{\F},\tilde{\Gamma}).
\end{eqnarray*}
\end{proof}

Consider a divisor of separatrices $\B= \sum_{B \in \sep(\F)} a_{B} \cdot B$ and
a blow-up $\pi :(\tilde{\C^{2}}, D  ) \to (\C^{2}, 0)$. We define the
strict transform  of $\B$ by $\pi$ in a standard manner:
\[\tilde{\B} =\pi^{*}\B =\sum_{B \in \sep(\F)} a_{B} \cdot \tilde{B}  ,\]
where $\tilde{B}  = \pi^*B$. We can germify this object at a fixed point $q \in D$,
which adds up to  erasing  all curves $\tilde{B}$ not passing through $q$.
This is a divisor of separatrices for $\F$ at $q$, that we denote by $\tilde{B}_{q}$.

Let $\Gamma$ be a formal branch at $(\C^{2},0)$, non-invariant by $\F$.
Denote by $\mathcal{I}(\Gamma)$ its set of infinitely near points.
 We define the \emph{tangency excess} of $\F$ \emph{along} $\Gamma$ as the integer
\begin{equation}
\label{eq-excees-tang}
\tau_{0}(\F,\Gamma) = \sum_{q \in \mathcal{I}(\Gamma)} \tau_{q}(\tilde{\F}) \nu_{q}(\tilde{\Gamma}),
\end{equation}
where  $\tau_{q}$ and  $\nu_{q}$  are calculated for the strict transforms of $\F$ and $\Gamma$ by the
sequence of blow-ups that produces $q$. Note that the  non-invariance of $\Gamma$   implies that the sum has a finite number of non-zero terms. If $\Gamma$ is a reduced curve, without $\F$-invariant components,  we write its decomposition in irreducible
components $\Gamma = \cup_{i=1}^{k}C_{i}$ and define additively
\[ \tau_{0}(\F,\Gamma) = \sum_{i=1}^{k} \tau_{0}(\F,C_{i}).\]

The following result  is a generalization  of \cite[Cor. 1]{cano2019}:
\begin{lem}
\label{lem-indicdicbald}
Let  $\F$ be  a germ of complex analytic foliation at $(\mathbb{C}^2, 0)$. Let  $\B$ be a balanced divisor of separatrices for $\F$  and $\Gamma$ be a non-invariant formal branch. Then
\[  (\mathcal{B}, \Gamma)_{0} =  \tg_0(\mathcal{F}, \Gamma)+1 - \tau_{0}(\F,\Gamma). \]
\end{lem}
\begin{proof}
Let us define
\[ \kappa_{0}(\F,\Gamma) =  \tg_0(\F, \Gamma)+1 - (\B, \Gamma)_{0}.\]
We will first see how $\kappa_{0}(\F,\Gamma)$ is modified by a blow-up   $\pi:(\tilde{\C^{2}},  D) \to (\C^{2}, 0)$   at the origin.
As before, denote by 	$\tilde{\F} = \pi^{*} \F$, $\tilde{\Gamma}  = \pi^{*} \Gamma$ and
  $\tilde{B}  = \pi^{*} B$ the strict transforms
 of $\F$, $\Gamma$ and $B  \in \sep(\F )$. If $q = \tilde{\Gamma}  \cap  D$,
let $\tilde{\B}_{q}  =  ( \pi ^{*} \B)_{q}$ denote the germ at $q$ of the
strict transform of the balanced divisor $\B$.
The germ    $D_{q}$ of  $D$ at $q$  is a separatrix of $\tilde{\F}$ at
 $q$ if and only if $\pi$ is a non-dicritical blow-up. Only in this case it   should  be included  in
a balanced divisor for $\tilde{\F}$:
\[  \B_{q} =
 \tilde{\B}_{q} + \epsilon D_{q}, \]
 where
\[ \epsilon  = \begin{cases}
 1  \ \ \text{if $\pi$ is non-dicritical} \medskip \\
 0 \ \ \text{if $\pi$ is dicritical}
\end{cases}
\]
is a balanced divisor of separatrices
 for $\tilde{\F}$ at $q$ \cite[Lem. 3.11]{genzmer2018}.

 Noether's formula (see, for instance, \cite[Lem. 3.3.4]{casas2000}),
 applied to  $B  \in \sep(\F)$, gives
\[(  B, \Gamma )_{0}= \nu_0 ( B) \nu_0(\Gamma) + ( \tilde{B},\tilde{\Gamma})_{q}. \]
By linearity, this formula can be extended to the balanced divisor
$\B$:
\[(\B, \Gamma)_{0} = \nu_{0}(\B) \nu_{0}(\Gamma) +  (\tilde{\B}_{q}, \tilde{\Gamma})_{q}.\]
 Proposition \ref{prop-mult-balanced} applied to this expression, gives
\begin{equation}
\label{eq-intersection-balanced}
  (\B, \Gamma)_{0} =  (\nu_{0}(\F)+1 - \tau_{0}(\F)) \nu_{0}(\Gamma) +  (\tilde{\B}_{q}, \tilde{\Gamma}  )_{q}.
\end{equation}
On the other hand,
\[\epsilon \nu_0(\Gamma)+ (\tilde{\B}_{q}, \tilde{\Gamma} )_{q} =
\epsilon (  D_{q}, \tilde{\Gamma} )_{q}+ ( \tilde{\B}_{q}, \tilde{\Gamma})_{q} =
 (\tilde{\B}_{q}  + \epsilon D_{q} , \tilde{\Gamma} )_{q} = (\B_{q}  , \tilde{\Gamma} )_{q} .\]
Inserting this in \eqref{eq-intersection-balanced},  we find
\begin{equation}
\label{eq-intesection-B-gamma}
(\mathcal{B}, \Gamma)_{0}  = ( \nu_{0}(\F) + 1 - \epsilon - \tau_{0}(\F)) \nu_{0}(\Gamma) +  (\B_{q}, \tilde{\Gamma}  )_{q}
= (m -\tau_{0}(\F) )\nu_{0}(\Gamma) +  (\B_{q}, \tilde{\Gamma}  )_{q},
\end{equation}
where $m$ is as \eqref{eq-def-m}.
From Proposition \ref{prop-blowup-tau} and  \eqref{eq-intesection-B-gamma}, we have
\begin{eqnarray}
\label{eq-kappa}
\kappa_{0}(\F,\Gamma) - \kappa_{q}(\tilde{\F},\tilde{\Gamma}) & = &
 (\tg_0(\mathcal{F}, \Gamma)+1 - (\mathcal{B}, \Gamma)_{0}) -
 (\tg_q(\tilde{\F}, \tilde{\Gamma})+1 - (\B_{q}, \tilde{\Gamma})_{q}) \smallskip   \\
  & =  & (\tg_0(\mathcal{F}, \Gamma) -  \tg_q(\tilde{\F}, \tilde{\Gamma}))
  -  ((\mathcal{B}, \Gamma)_{0}) - (\B_{q}, \tilde{\Gamma})_{q}) \smallskip  \nonumber \\
  & = & m \nu_{0}(\Gamma) - (m -\tau_{0}(\F) )\nu_{0}(\Gamma) =    \tau_{0}(\F) \nu_{0}(\Gamma) \nonumber
\end{eqnarray}

Let $q_{0}, q_{1}, \cdots , q_{n}$ be successive points in  $\mathcal{I}(\Gamma)$, where
$q_{0}= 0 \in \C^{2}$.
 Let us denote by $\tilde{\F}_{j}$ and by $\tilde{\Gamma}_{j}$ the strict transforms of $\F$ and
$\Gamma$ at $q_{j}$,   with the convention that  $\F = \tilde{\F}_{0}$ and  $\Gamma = \tilde{\Gamma}_{0}$.
We can obtain $n$ such that, at  $q_{n}$,  both
$\tilde{\Gamma}_{n}$ and $\tilde{\F}_{n}$ are regular and transversal, implying that
$\kappa_{q_{n}}(\tilde{\F}_{n},\tilde{\Gamma}_{n})= 0$.
Indeed,   we can fix local coordinates $(x,y)$ at $q_{n}$ such that $\tilde{\F}_{n}$
is defined by $\omega=dx$ and   $\tilde{\Gamma}_{n}$ is the $x$-axis.
Thus, $B_{q_{n}}  = \left\{ x=0 \right\}$ is the only separatrix of $\tilde{\F}_{n}$, so
  that the balanced divisor is  $\B_{q_{n}}  = B_{q_{n}}$. It is then  straightforward that
$\tg_q(\tilde{\F}, \tilde{\Gamma})= 0$ and that
 $(\B_{q}, \tilde{\Gamma})_{q} = (B, \Gamma)_{q} = 1$.

Now, we have
 \begin{eqnarray*}
 \tg_0(\F, \Gamma)+1 - (\B, \Gamma)_{0}& = & \kappa_{0}(\F,\Gamma) \\
 & = &  \kappa_{q_{0}}(\tilde{\F}_{0},\tilde{\Gamma}_{0}) - \kappa_{q_{n}}(\tilde{\F}_{n},\tilde{\Gamma}_{n}) \\
 & = & \sum_{j=0}^{n-1} \left( \kappa_{q_{j}}(\tilde{\F}_{j},\tilde{\Gamma}_{j}) - \kappa_{q_{j+1}}(\tilde{\F}_{j+1},\tilde{\Gamma}_{j+1}) \right) \\
 & = & \sum_{j=0}^{n-1}  \tau_{j}(\tilde{\F}_{j}) \nu_{j}(\tilde{\Gamma}_{j}) = \tau_{0}(\F,\Gamma)
\ \ \text{(by \eqref{eq-kappa} and \eqref{eq-excees-tang})   },
\end{eqnarray*}
proving the Lemma.
\end{proof}

Now we can prove the main result of this section:

\begin{prop}
\label{prop-polar-mu-nu}
Let $\F$ be a germ of singular complex analytic foliation at $(\mathbb{C}^2, 0)$ and $\B$ be a balanced divisor
of separatrices for $\F$. Then
	\begin{equation*}
	p_0(\F,  \B) =\mu_0(\F)+\nu_0(\F)  - \tau_{0}(\F,\Gamma) ,
	\end{equation*}
where $\Gamma$ is a generic polar curve.
\end{prop}
\begin{proof}
Let  	$\Gamma =\Po_{(a:b)}$ be a generic polar curve for $\F$.
 We can suppose that $a=1$.
 Write $\Gamma=\cup_{i=1}^{k}C_{i}$ the decomposition of $\Gamma$ in irreducible components.
As a consequence of Lemma \ref{lem-indicdicbald}, we have
\begin{eqnarray}
\label{eq-inequality-p0}
p_0(\F,  \B )  \ =  \ (\Gamma, \B)_{0} \ = \ \sum_{i=1}^{k}  (C_{i}, \B)_{0}&  = & \sum_{i=1}^{k} (\tg_0(\F, C_{i})+1 - \tau_0(\F, C_{i}) )\\
  & = & \sum_{i=1}^{k} \left(\tg_0(\F, C_{i})+1 \right) - \tau_0(\F, \Gamma).   \nonumber
\end{eqnarray}

From this point on, the  proof is the same  as that of	\cite[Prop. 2]{cano2019}.
We write it here for the sake of completeness.
For each $C_{i}$, we write a    parametrization
 $ \gamma_{i}(t)=(x_i(t),y_i(t))$. Thus, from the
 equation $P-b Q=0$  of   $\Gamma$, we find that
 $P(\gamma_i(t))=b Q(\gamma_i(t))$ for all $t$.
We have
\begin{eqnarray}
\label{eq-tg-mu-nu}
		\tg_0(\mathcal{F}, C_{i}) + 1  & = & \ord_{t=0} \left( \gamma_{i}^*\omega  \right) + 1 \\
 &=&
 \ord_{t=0} \left\{  Q(\gamma_i(t)) \left(  by_{i}'(t) -  x_{i}'(t)
		  \right) \right\} + 1  \nonumber \\
	&=& \ord_{t=0}   Q(\gamma_i(t))  +\ord_{t=0} \left(   by_{i}'(t) -  x_{i}'(t) \right)+ 1 \nonumber  \\
		&=& \ord_{t=0}   Q(\gamma_i(t))    +
\ord_{t=0} \left(   by_{i}(t) -  x_{i}(t)
		    \right) \nonumber
		 \\
		&=& (Q, C_{i})_{0}   +\nu_{0}(C_{i}). \nonumber
\end{eqnarray}
The proof  is then completed by putting \eqref{eq-tg-mu-nu} in \eqref{eq-inequality-p0},
observing that
\[\sum_{i=1}^{k}  (Q, C_{i})_{0} = (Q, \Gamma)_{0} =
(Q, P-bQ)_{0} = (Q,P)_{0}= \dim_{\mathbb{C}} \C\{x,y\}/(P, Q)=\mu_0(\mathcal{F}) \]
and  that $\sum_{i=1}^{k} \nu_{0}(C_{i}) =  \nu_{0}(\Gamma)$.
\end{proof}

\section{The Milnor number and separatrices}
\label{section-theoremB}

In this section,
we resume the study of the pair   $\R$-analytic vector field and its complexification, $X_{\R}$ and $X_{\C}$, along with their associated foliations,
$\F_{\R}$ and $\F_{\C}$. Hence,
the subscripts $\R$ or $\C$ will be reincorporated to our notation.

Let $\Gamma_{\R}$ be a germ of real analytic curve,
non-invariant by $\F_{\R}$.  We can define, in an obvious manner, a real version of
tangency excess of a foliation
along $\Gamma_{\R}$. If $\Gamma_{\R}$ is a branch,  $\tau_{0}(\F,\Gamma_{\R})$ is   the integer
 obtained, in formula \eqref{eq-excees-tang}, by  simply replacing $\tau_{q}(\tilde{\F}_{\C})$  by
 its real equivalent $\tau_{q}(\tilde{\F}_{\R})$. If $\Gamma_{\R}$ is a union of branches,
 $\tau_{0}(\F,\Gamma_{\R})$ is defined additively, as we did before.
Thus, if $\Gamma_{\R}$ is a germ of $\R$-analytic curve, without $\F_{\R}$-invariant branches, and
$\Gamma_{\C}$ is its complexification,  we have, by \eqref{eq-tau-mod2},
\begin{equation}
\label{eq-tau-noninv-mod2}
\tau_{0}(\F_{\R},\Gamma_{\R}) \equiv \tau_{0}(\F_{\C},\Gamma_{\C}) \pmod{2} .
\end{equation}

We chose, for $\F_{\C}$, a generic polar curve
$\Gamma_{\C} = \Po_{(a:b)}$ with $a, b \in \R$. Its equation
$aP_{\C} - bQ_{\C}=0$ has real coefficients and its decomposition
in irreducible components is such that the components which are
not $J$-symmetric appear in pairs. Let $\Gamma_{\R}$   be the real curve corresponding to
 the union of
$J$-symmetric components of $\Gamma_{\C}$. We will call it a \emph{generic real polar curve}
for $\F_{\R}$. We set
 \[ p_{0}(\F_{\R},\B_{\R}) = (\Gamma_{\R}, \B_{\R})_{0} ,\]
where $\B_{\R}$ is  balanced divisor of separatrices for $\F_{\R}$ and the intersection number
 should be understood as that of complexified  curves.
Again, it is a consequence of the $J$-symmetry of $\F_{\C}$ that
\begin{equation}
\label{eq-polar-mult-mod2}
 p_{0}(\F_{\R},\B_{\R}) \equiv p_{0}(\F_{\C},\B_{\C}) \pmod{2}.
 \end{equation}

Considering these definitions and equations
\eqref{eq-tau-noninv-mod2} and \eqref{eq-polar-mult-mod2}, we can
state  the following real version for
Proposition \ref{prop-polar-mu-nu}:
\begin{prop}
\label{prop-polar-mu-nu-mod2}
Let $\F_{\R}$ be a germ of real analytic foliation at $(\R^2, 0)$. If    $\B_{\R}$ is a balanced divisor
of separatrices, then
\begin{equation*}
p_{0}(\F_{\R},\B_{\R}) = \mu_0(\F) + \nu_0(\F)  - \tau_{0}(\F,\Gamma_{\R}) \pmod{2},
\end{equation*}
where $\Gamma_{\R}$ is a generic real polar curve.
\end{prop}
We can assure the existence of a separatrix for $\F_{\R}$   by showing that
$p_{0}(\F_{\R},\B_{\R})$ is non-zero. Indeed, this implies that $\B_{\R}$
has non-empty support. In this way, we can
use Proposition \ref{prop-polar-mu-nu-mod2} to prove   the following result:

\begin{teorB}
Let $X_{\R}$ be a germ of real analytic vector field at $(\R^{2},0)$ with associated foliation
$\F_{\R}$. Suppose that $\F_{\R}$ is a real topological generalized curve foliation.
If $\mu_{0}(\F_{\R})$ is even,
then $\F_{\R}$ has a formal separatrix.
\end{teorB}
\begin{proof}
If $\nu_{0}(\F_{\R})$ is even, there is a separatrix by Theorem A.
Thus, we can assume that $\nu_{0}(\F_{\R})$ is odd.
The foliation $\F_{\R}$ is a topological $\R$GC, so
 there are no topological saddle-nodes in its reduction of singularties. Thus,
 all tangency excess indices $\tau_{q}(\tilde{\F}_{\R})$ for the infinitely near
points of the real polar curve $\Gamma_{\R}$ are even, giving that     $\tau_{0}(\F,\Gamma_{\R})$ is also even.
Now, we find by Proposition \ref{prop-polar-mu-nu-mod2} that
$p_{0}(\F_{\R},\B_{\R}) =1 \pmod{2}$, which implies the existence of a separatrix.
\end{proof}
With arguments similar to those of Corollaries \ref{cor-nu-RGC} and \ref{cor-nu-centerfocus}, we obtain
the following results of \cite{Risler2001}:

\begin{cor} If  $\F_{\R}$ is an $\R$-generalized curve foliation and $\mu_{0}(\F_{\R})$ is even,
then  it admits an
analytic separatrix.
\end{cor}

\begin{cor}
If  $\F_{\R}$ is a center-focus foliation, then $\mu_{0}(\F_{\R})$ is odd.
\end{cor}

The following example  shows that, with the hypothesis stated, we cannot
  ask for   analytic separatrices in  Theorems A and B:
\begin{example}
\label{ex-mol-sanz} {\rm
	Let $a(x) \in   \mathbb{R}\{x\}$ be a convergent series in one variable such that   $a(0)=a'(0)=0$,  and consider the $\R$-analytic vector field
\[X_{\R} = (y^{2}+x^{4}) \frac{\partial}{\partial x} +\left(-xy +x^{3}a(x) + \frac{a(x)}{x}y^{2}\right) \frac{\partial}{\partial y}.\]
It is proved in \cite[Prop. 10]{molsanz2019} that $X_{\R}$ has a unique separatrix which is, for a generic choice of $a(x)$,   purely formal.
By a first blow-up $\pi_{1}:(x,y_{1})   \mapsto   (x,y = xy_{1})$,
the strict transform foliation has a unique singularity over $D_{1} = \pi_{1}^{-1}(0)$, placed at  $(x,y_{1})=(0,0)$, which is a saddle-node
with strong separatrix contained in the   $D_{1}$.
Another blow-up $\pi_{1}:(x,y_{2})   \mapsto   (x,y_{1} = xy_{2})$ produces
a strict transform foliation induced at the origin by a vector field of the form
\[  x^{3}  \frac{\partial}{\partial x}
+ (- y_{2}(1 + b(x,y_{2})) + a(x))  \frac{\partial}{\partial y_{2}},\]
where  $b(x,y_{2})$ contains terms of order at least two.
The weak index of this saddle-node is $\iota^{w} = 3$ and it is not a topological saddle-node (actually,
 it is a topological saddle, that can be converted into a topological node by changing $y$ to $-y$ in the
 coefficients of $X_{\R}$).
The vector field $X_{\R}$ is thus a topological real generalized curve whose only separatrix (for a
 generic choice of $a(x)$) is purely formal.
It is easy to see that $\nu_{0}(X_{\R}) = 2$ and  $\mu_{0}(X_{\R})=6$, showing that the statements of Theorems
A and B cannot be improved by asking a convergent separatrix.
}\end{example}

Our main results could also have been stated in terms of the more general family of \emph{topological second type foliations}: real foliations without tangent topological saddle nodes.
We  finish by  mentioning this fact, placing it as a consequence of Theorems A and B:

\begin{cor} Let $\F_{\R}$ be a germ of  topological  $\R$-second type foliation at $(\R^{2},0)$.
If either $\nu_{0}(\F_{\R})$ or $\mu_{0}(\F_{\R})$ is even,
then $\F_{\R}$ has a formal separatrix.
\end{cor}
\begin{proof}
It there were no separatrices, the reduced foliation $\tilde{\F}_{\R}$ would have no trace singularities and, so, $\F_{\R}$ would be a topological  $\R$-generalized curve foliation.
This would contradict Theorems A and B.
\end{proof}

\bibliographystyle{plain}
\bibliography{referencias}

\end{document}